\documentclass[11pt]{amsart}
\usepackage{amsmath,amssymb,latexsym,soul,cite}
\usepackage{color,enumitem,graphicx}
\usepackage[colorlinks=true,urlcolor=blue,citecolor=red,
linkcolor=blue,linktocpage,pdfpagelabels,bookmarksnumbered,bookmarksopen]{hyperref}
\usepackage[english]{babel}
\usepackage[T1]{fontenc}
\usepackage[hyperpageref]{backref}

\usepackage[left=3.5cm,right=3.5cm,top=2.9cm,bottom=2.9cm]{geometry}

\numberwithin{equation}{section}

\newtheorem{thm}{Theorem}[section]
  \theoremstyle{plain}
  \newtheorem{lem}[thm]{Lemma}
  \theoremstyle{plain}
  \newtheorem{prop}[thm]{Proposition}
  \theoremstyle{plain}
  
  \theoremstyle{remark}
  \newtheorem{rem}[thm]{Remark}

\newcommand{\N}{{\mathbb N}}

\newcommand{\R}{{\mathbb R}}
\newcommand{\eps}{\varepsilon}

\renewcommand{\le}{\leqslant}
\renewcommand{\ge}{\geqslant}
\renewcommand{\leq}{\leqslant}
\renewcommand{\geq}{\geqslant}
\def\Xint#1{\mathchoice
{\XXint\displaystyle\textstyle{#1}}%
{\XXint\textstyle\scriptstyle{#1}}%
{\XXint\scriptstyle\scriptscriptstyle{#1}}%
{\XXint\scriptscriptstyle\scriptscriptstyle{#1}}%
\!\int}
\def\XXint#1#2#3{{\setbox0=\hbox{$#1{#2#3}{\int}$ }
\vcenter{\hbox{$#2#3$ }}\kern-.57\wd0}}
\def\intmed{\Xint-}
\DeclareMathOperator*{\essliminf}{ess\, liminf}

\DeclareMathOperator*{\esssup}{ess\, sup}
\DeclareMathOperator*{\essinf}{ess\, inf}

\title{$H^s$ versus $C^0$-weighted minimizers}

\author[A.\ Iannizzotto]{Antonio Iannizzotto}
\author[S.\ Mosconi]{Sunra Mosconi}
\author[M.\ Squassina]{Marco Squassina}

\address{Dipartimento di Informatica
\newline\indent
Universit\`a degli Studi di Verona
\newline\indent
Strada Le Grazie I-37134 Verona, Italy}
\email{marco.squassina@univr.it}
\email{antonio.iannizzotto@univr.it}

\address{Dipartimento di Matematica e Informatica
\newline\indent
Universit\`a degli Studi di Catania
\newline\indent
Viale A. Doria 6 I-95125 Catania, Italy}
\email{mosconi@dmi.unict.it}

\thanks{The first and second authors were supported by GNAMPA project: ``{\em Problemi al contorno per operatori non locali non lineari}''. The third author was supported by MIUR project: ``{\em Variational and topological methods in the study of nonlinear phenomena}''. This work was partially carried out during a stay of S.\ Mosconi at the University of Verona. He would like to express his gratitude to the Department of Computer Science for the warm hospitality.}
\subjclass[2000]{35P15, 35P30, 35R11}
\keywords{Fractional Laplacian, fractional Sobolev spaces, local minimizers.}

\begin{document}

\begin{abstract}
We study a class of semi-linear problems involving the fractional Laplacian under subcritical or critical growth assumptions. We prove that, for the corresponding functional, local minimizers with respect to a $C^0$-topology weighted with a suitable power of the distance from the boundary are actually local minimizers in the natural $H^s$-topology.
\end{abstract}

\maketitle


\section{Introduction and main result}

\noindent
Let $\Omega$ be a bounded domain in $\R^N$, $N \ge 2$, with $C^{1,1}$ boundary $\partial\Omega$, and $s \in (0,1)$. 
We consider the following boundary value problem driven by the fractional Laplacian operator
\begin{equation} \label{bvp}
\begin{cases}
(- \Delta)^s\, u = f(x,u) & \text{in } \Omega \\
u = 0 & \text{in } \R^N \setminus \Omega.
\end{cases}
\end{equation}
The fractional Laplacian operator is defined by
\[
(-\Delta)^su(x):=C(N,s)\lim_{\eps\searrow 0}\int_{\R^N\setminus B_\eps(x)}\frac{u(x)-u(y)}{|x-y|^{N+2s}}\,dy, \quad\,\, x\in\R^N,
\]
where $C(N,s)$ is a suitable positive normalization constant.\ The nonlinearity $f:\Omega\times\R\to\R$ is a Carath\'eodory mapping which satisfies the growth condition
\begin{equation}\label{gc}
|f(x,t)|\le a(1+|t|^{q-1}) \,\,\, \mbox{a.e.\ in $\Omega$ and for all $t\in\R$ \,\,\, ($a>0$, $1\le q\le 2^*_s$)}
\end{equation}
(here $2^*_s:=2N/(N-2s)$ is the fractional critical exponent). Condition \eqref{gc} is referred to as a {\em subcritical} or {\em critical} growth if $q<2^*_s$ or $q=2^*_s$, respectively.
\vskip2pt
\noindent
For existence and multiplicity results for problem \eqref{bvp} via variational methods, see \cite{SV, SV1, SV2, SV3}. 
Concerning regularity and non-existence of solutions, we refer the reader to \cite{CS,RS,RS1,RS2,CS1,CS2} and to the references therein. 
Although the fractional Laplacian operator $(-\Delta)^s$, and more generally pseudodifferential operators, have been a classical topic of functional analysis since long ago,
the interest for such operator has constantly increased in the last few years.
Nonlocal operators such as $(-\Delta)^s$ naturally arise
in continuum mechanics, phase transition phenomena,
population dynamics and game theory, as they are the typical outcome of stochastical stabilization of L\'evy processes, see e.g.\ 
the work of Caffarelli \cite{C} and the references therein. 
\vskip2pt
\noindent
Problem \eqref{bvp} admits a variational formulation. For any measurable function $u:\R^N\to\R$ we define the Gagliardo seminorm by setting
\[
[u]_{s}^2:=\int_{\R^{2N}}\frac{(u(x)-u(y))^2}{|x-y|^{N+2s}}\, dx\, dy,
\]
and we introduce the fractional Sobolev space
\[
H^s(\R^N)=\{u\in L^2(\R^N):\,[u]_{s}<\infty\},
\]
which is a Hilbert space. We also define a closed subspace
\begin{equation}
\label{defX}
X(\Omega)=\{u\in H^s(\R^N):\,u=0\,\,\,\mbox{a.e. in $\R^N\setminus\Omega$}\}.
\end{equation}
Due to the fractional Sobolev inequality, $X(\Omega)$ is a Hilbert space with inner product
\begin{equation}
\label{innprod}
\langle u,v\rangle_X=\int_{\R^{2N}}\frac{(u(x)-u(y))(v(x)-v(y))}{|x-y|^{N+2s}}\, dx\, dy,
\end{equation}
which induces a norm $\|\cdot\|_X=[\,\cdot\,]_{s}$. 
Set for all $u\in X(\Omega)$
\[
\Phi(u):=\frac{\|u\|_X^2}{2}-\int_\Omega F(x,u)\, dx,
\]
where
\[
F(x,t)=\int_0^t f(x,\tau)\,d\tau,\quad x\in\Omega,\,\, t\in{\mathbb R}.
\]
Then, $\Phi\in C^1(X(\Omega))$ and all its critical points are (up to a normalization constant depending on $s$ and $N$, which we will neglect henceforth)
weak solutions of \eqref{bvp}, namely they satisfy 
\begin{equation}\label{weak}
\langle u,v\rangle_X=\int_\Omega f(x,u)v\, dx, \qquad \text{for all $v\in X(\Omega)$}.
\end{equation}
In the framework of variational methods, local minimizers of the energy $\Phi$ play a fundamental r\^ole. In a number of situations, one singles out particular solutions arising as constrained minimizers of the energy functional in order-defined subsets of $X(\Omega)$. Since usually the latters have empty interior, it is a nontrivial task to prove that such constrained minimizers are actually unconstrained local minimizers of the energy in the whole $X(\Omega)$. 
\vskip2pt
\noindent
This issue was analyzed by Brezis \& Nirenberg \cite{BN} for the semilinear problem
\begin{equation}
\label{probsem}
\begin{cases}
-\Delta u=f(x,u) & \text{in $\Omega$} \\
u=0 & \text{on $\partial\Omega$}.
\end{cases}
\end{equation}
They observe that the $C^1(\overline\Omega)$ topology gives rise to nonempty interiors for most of such order-defined subsets. By the Hopf lemma, constrained minimizers solutions can be seen to lie in the $C^1$-interior of the constraint set. The key point which they proved is that local minima with respect to the $C^1$-topology remain so in the $H^1$-one, despite the latter being much weaker than the former. Thus the constrained minimization procedure gives rise to solutions {\em which are also unconstrained local minimizers}.
\noindent
This method was not only fruitfully applied to obtain a huge number of multiplicity results for the semilinear problem \eqref{probsem}, but also extended to cover a wide range of variational equations.
\vskip2pt
\noindent
In the present paper, we aim to develop all the tools needed to reproduce this technique in the fractional setting. 
In doing so we will gather a number of more or less known results for the fractional Laplacian, including weak and strong maximum principles, a Hopf lemma, and {\em a priori} estimates for the weak solution of problems of the type \eqref{bvp}. We will provide a proof for those results for which only a statement was available, or strengthen the conclusions with respect to existing literature. In some cases, we will generalize results known only for special cases such as linear problems, eigenvalue problems, or positive solutions. Detailed discussion will be made for each result. 
We will then prove that being a local minimizer for $\Phi$ with respect to a suitable weighted $C^0$-norm, is equivalent to being an $X(\Omega)$-local minimizer. Particular attention will be paid to the {\em critical case}, i.e., $q=2^*_s$ in \eqref{gc}, which presents a twofold difficulty: a loss of compactness which prevents minimization of $\Phi$, and the lack of uniform {\em a priori} estimates for the weak solutions of \eqref{bvp}. Finally we will give three different applications of this result to nonlocal semilinear problem. A sub-supersolution principle for local minimizers, a multiplicity result for singular nonlinearities, and a multiplicity result for smooth ones.
\vskip2pt
\noindent
In order to state the local minimization result, we now describe the natural topology corresponding the $C^1$-one above. Define $\delta:\overline\Omega\to\R_+$ by 
\begin{equation}
\label{defdelta}
\delta(x):={\rm dist}(x,\R^N\setminus\Omega),\quad x\in\overline\Omega,
\end{equation}
and consider problem \eqref{bvp} with $f(x, u)=f(x)$ and $f\in L^\infty(\Omega)$.
Ros Oton \& Serra in \cite{RS} proved that a solution $u$ to \eqref{bvp} is such that $u/\delta^s\in C^{\alpha}(\overline\Omega)$.
Thus, a natural topology for the fractional problem \eqref{bvp} seems to be the one of 
\[
C^0_\delta(\overline\Omega)=\Big\{u\in C^0(\overline{\Omega}): \,\mbox{$\displaystyle\frac{u}{\delta^s}$ admits a continuous extension to $\overline{\Omega}$}\Big\}
\]
with norm $\| u\|_{0, \delta}=\|u/\delta^s\|_\infty$. 
Our main result establishes that indeed local minimizers of $\Phi$ in $C^0_\delta(\overline\Omega)$ and in $X(\Omega)$ coincide:

\begin{thm}
\label{min-local}
Let $\Omega$ be a bounded $C^{1,1}$ domain, $f:\Omega\times\R\to\R$ a Carath\'eodory function satisfying \eqref{gc}, and $u_0\in X(\Omega)$. Then,
the following assertions are equivalent:
\vskip2pt
\noindent
\ $(i)$\ \  there exists $\rho>0$ such that
$\Phi(u_0+v)\ge\Phi(u_0)$
for all $v\in X(\Omega)\cap C^0_\delta(\overline\Omega)$, $\|v\|_{0,\delta}\le\rho$,
\vskip2pt
\noindent
\ $(ii)$\    there exists $\eps>0$ such that
$\Phi(u_0+v)\ge\Phi(u_0)$
for all $v\in X(\Omega)$, $\|v\|_X\le\eps$.
\end{thm}

\noindent
Notice that, contrary to the result of \cite{BN} in the local case $s=1$, there is no relationship between the topologies 
of $X$ and $C^0_\delta(\overline\Omega)$.
\vskip2pt

\noindent
The paper has the following structure: in Section \ref{s2} we establish some preliminary results, including the weak and strong maximum principles, and a fractional Hopf lemma; in Section \ref{apb} we prove {\em a priori} bounds for non-local problems, both in the subcritical and the critical cases; in Section \ref{s3} we prove Theorem \ref{min-local}; in Section \ref{s4} we give some applications of our main result; and in Section \ref{sec6} we discuss possible extensions and developments.

\begin{rem}\label{barrios}
After completing the present work, we became aware of an interesting paper of Barrios, Colorado, Servadei \& Soria \cite{BCSS}, where a special case of Theorem \ref{min-local} is obtained and used to study fractional boundary value problems involving pure power type nonlinearities with critical growth.
\end{rem}

\section{Preliminary results}\label{s2}

\noindent
In this section we will state and prove some basic results about weak (super)solutions of non-local boundary value problems. 
\vskip2pt
\noindent
For $\delta$ as in \eqref{defdelta}, we define the weighted H\"older-type spaces ($\alpha\in(0,1)$)
\begin{equation}
\label{defcdelta}
\begin{split}
C^0_\delta(\overline\Omega) &:=\Big\{u\in C^0(\overline{\Omega}): \,\mbox{$\displaystyle\frac{u}{\delta^s}$ admits a continuous extension to $\overline{\Omega}$}\Big\}, \\
C_\delta^{0,\alpha}(\overline\Omega)&:=\Big\{u\in C^0(\overline{\Omega}):\,\mbox{$\displaystyle\frac{u}{\delta^s}$ admits a $\alpha$-H\"older continuous extension to $\overline{\Omega}$}\Big\},
\end{split}
\end{equation}
endowed with the norms
\[\|u\|_{0,\delta}:=\Big\|\frac{u}{\delta^s}\Big\|_\infty, \,\,\,\, \|u\|_{\alpha,\delta}:=\|u\|_{0,\delta}+\sup_{x,y\in\overline\Omega,\,x\ne y}\frac{|u(x)/\delta(x)^s-u(y)/\delta(y)^s|}{|x-y|^\alpha},\]
respectively. Clearly, any function $u\in C^0_\delta(\overline\Omega)$ vanishes on $\partial\Omega$, so it can be naturally extended by $0$ on $\R^N\setminus\overline\Omega$. In this way, we will always consider elements of $C^0_\delta(\overline\Omega)$ as defined on the whole $\R^N$. Moreover, by virtue of Ascoli's theorem, the embedding $C_\delta^{0,\alpha}(\overline\Omega)\hookrightarrow C_\delta^{0}(\overline\Omega)$ is compact. 
\vskip2pt
\noindent
The Hilbert space $X(\Omega)$ has been defined in \eqref{defX}, with inner product \eqref{innprod}. The embedding $X(\Omega)\hookrightarrow L^q(\Omega)$ is continuous for all $q\in [1,2^*_s]$ and compact if $q\in[1,2^*_s)$ (see \cite[Theorem 7.1]{DPV}). We will set
\[X(\Omega)_+=\{u\in X(\Omega):\, u\ge 0\,\mbox{a.e. in $\Omega$}\},\]
the definition of $H^s(\R^N)_+$ being analogous. For all $t\in\R$ we set
\[t_\pm=\max\{\pm t,0\}.\]
Besides, for all $x\in\R^N$, $r>0$ we denote by $B_r(x)$ (respectively, $\overline B_r(x)$) the open (respectively, closed) ball of radius $r$ centered at $x$ in $\R^N$. Similarly, $B^X_\rho(u)$, $\overline B^X_\rho(u)$ ($B^\delta_\rho(u)$, $\overline B^\delta_\rho(u)$) will denote an open and a closed ball, respectively, in $X(\Omega)$ (in $C^0_\delta(\overline\Omega)$) centered at $u$ with radius $\rho$. Finally, $C$ will denote a positive constant whose value may change case by case.
\vskip2pt
\noindent
We consider the following linear equation with general Dirichlet condition:
\begin{equation}
\label{dirichlet}
\begin{cases}
(-\Delta)^s u=f &\text{in $\Omega$}\\
u=g&\text{in $\R^N\setminus\Omega$,}
\end{cases}
\end{equation}
where $f\in L^\infty(\Omega)$ and $g\in H^s(\R^N)$. We say that $u\in H^{s}(\R^N)$ is a {\em weak supersolution} of \eqref{dirichlet} if 
$u\geq g$ a.e.\ in $\R^N\setminus \Omega$ and the following holds for all $v\in X(\Omega)_+$:
\[\int_{\R^{2N}}\frac{(u(x)-u(y))(v(x)-v(y))}{|x-y|^{N+2s}} \, dx\, dy \geq \int_\Omega fv\,dx.\]
The definition of a {\em weak subsolution} is analogous. Clearly, $u\in H^s(\R^N)$ is a weak solution of \eqref{dirichlet} if it is both a weak supersolution and a weak subsolution (this definition of a weak solution agrees with \eqref{weak}). These definitions will be used throughout the paper.
\vskip2pt
\noindent
From \cite[proof of Theorem 1.1, Remark 4.2]{DKP} we have the following bound.

\begin{thm}
\label{dkp}
Let $u\in H^s(\R^N)$ be a weak subsolution of \eqref{dirichlet} with $f=0$. Then, there exists a constant $C=C(N, s)$ such that for any $k\in \R$, $x_0\in\Omega$, $r>0$ such that $B_r(x_0)\subseteq\Omega$, we have
\[\esssup_{B_{r/2}(x_0)} u\leq k+{\rm Tail} ((u-k)_+; x_0, r/2)+C\left(\intmed_{B_r(x_0)} (u-k)_+^2 \, dx \right)^{\frac 1 2},\]
where the {\em nonlocal tail} of $v\in H^s(\R^N)$ at $x_0$ is defined by
\[
{\rm Tail} (v; x_0, r):=r^{2s}\int_{\R^N\setminus B_r(x_0)}\frac{|v(x)|}{|x-x_0|^{N+2s}}\, dx.
\]
\end{thm}

\noindent
The following lemma follows slightly 
modifying the proof of \cite[Lemma 3.2]{RS}:

\begin{lem}\label{subs}
If $0<r<R$, $f=0$, and $g\in H^s(\R^N)$ is such that
\[g(x)=\begin{cases}
1 & \text{if $x\in\overline B_r(0)$} \\
0 & \text{if $x\in\R^N\setminus B_R(0)$,}
\end{cases}\]
then there exist $c=c(r,R)>0$ and a weak solution $\varphi\in H^s(\R^N)$ of \eqref{dirichlet} with $f=0$ in the domain $B_R(0)\setminus\overline B_r(0)$, such that a.e. in $\R^N$
\[\varphi(x)\ge c(R-|x|)_+^s.\]
\end{lem}

\noindent
In the following sections we will use the following fundamental regularity estimate proved in \cite[Theorem 1.2]{RS}.

\begin{thm}
\label{ros}
Let $u$ be a weak solution of \eqref{dirichlet} with $f\in L^\infty(\Omega)$, $g=0$. Then there exist $\alpha\in (0, \min\{s, 1-s\})$ such that $u\in C^{0,\alpha}_\delta(\overline{\Omega})$ (see \eqref{defcdelta}) and $C=C(\Omega, N, s)$ such that
\[
\| u\|_{\alpha, \delta}\leq C\|f\|_\infty.
\]
\end{thm}

\noindent
We now prove a weak maximum principle for weak supersolutions of problem \eqref{dirichlet}. While the non-negativity result is well known, we could not find a statement of the semicontinuity property in the literature.

\begin{thm}\label{wmp}
If $u\in H^s(\R^N)$ is a weak supersolution of \eqref{dirichlet} with $f=0$ and $g\in H^s(\R^N)_+$, then $u\ge 0$ a.e. in $\Omega$ and $u$ admits a lower semi-continuous representative in $\Omega$. 
\end{thm}
\begin{proof}
First we prove that $u\in H^s(\R^N)_+$. Since $u\ge g\ge 0$ a.e. in $\R^N\setminus\Omega$, we have $u_-\in X(\Omega)_+$. So, the elementary inequality $(a-b)(a_--b_-)\le -(a_--b_-)^2$, $a, b\in \R$, yields
\[0\le\int_{\R^{2N}}\frac{(u(x)-u(y))(u_-(x)-u_-(y))}{|x-y|^{N+2s}} \, dx\, dy\le-[u_-]_{s}^2,\]
hence $u_-=0$, namely $u\in H^s(\R^N)_+$.
Now we find a lower semi-continuous function $u^*$ such that $u(x)=u^*(x)$ a.e. in $\Omega$. Set for all $x_0\in\R^N$
\[u^*(x_0)=\essliminf_{x\to x_0} u(x_0).\]
Since $u\in H^s(\R^N)_+$ we have $u^*\ge 0$ a.e. and $u^*$ is lower semi-continuous in $\Omega$. Now assume that $x_0\in\Omega$ is a Lebesgue point for $u$ and define $u(x_0)$ accordingly, noting that
\[
u(x_0):=\lim_{r\to 0^+}\intmed_{B_r(x_0)} u\, dx\geq \lim_{r\to 0^+}\essinf_{B_r(x_0)} u=u^*(x_0).
\]
To prove the reverse inequality, we apply Theorem \ref{dkp} to the function $-u$ (which is a weak subsolution of \eqref{dirichlet}) with $k=-u(x_0)$ and get
\[
\esssup_{B_{r/2}(x_0)}(-u) \leq -u(x_0)+{\rm Tail} ((u(x_0)-u)_+; x_0, r/2) +C\, \Big(\intmed_{B_r(x_0)} (u(x_0)-u(x))_+^2\, dx\Big)^{\frac 1 2} .
\]
Letting $r\to 0^+$, since $x_0$ is a Lebesgue point we have
\[
\lim_{r\to 0^+}\Big(\intmed_{B_r(x_0)} (u(x_0)-u(x))_+^2 \, dx\Big)^{\frac 1 2}=0.
\]
Besides, by the H\"older inequality we have
\begin{align*}
& {\rm Tail} ((u(x_0)-u)_+; x_0, r/2)  \\
&\le r^{2s}\Big(\int_{\R^N\setminus B_r(x_0)}\frac{(u(x_0)-u(x))_+^2}{|x_0-x|^{N+2s}}\, dx\Big)^\frac{1}{2}\Big(\int_{\R^N\setminus B_r(x_0)}\frac{1}{|x_0-x|^{N+2s}}\, dx\Big)^\frac{1}{2} \\
&\le Cr^s\Big(\int_{\R^N}\frac{(u(x_0)-u(x))^2}{|x_0-x|^{N+2s}}\, dx\Big)^\frac{1}{2}\to 0\quad \text{as $r\to 0^+$},
\end{align*}
since, being $u\in H^s(\R^N)$, the integral is finite for a.e. $x_0\in\Omega$. So we have
\[\lim_{r\to 0^+}\esssup_{B_{r/2}(x_0)}(-u)\le -u(x_0),\]
i.e. $u^*(x_0)\ge u(x_0)$ for a.e. Lebesgue point $x_0\in \Omega$ for $u$, and hence for a.e. $x_0\in \Omega$. 
\end{proof}

\noindent
Henceforth any weak supersolution to \eqref{dirichlet}, with $f=0$, will be identified with its 
lower semi-continuous regularization, and any weak subsolution with its upper semi-continuous 
regularization, so that their value at any point is well defined.
\vskip2pt
\noindent
By means of Theorem \ref{wmp} and Lemma \ref{subs} we can prove the following strong maximum principle. 

\begin{thm}\label{smp}
If $u\in H^s(\R^N)\setminus\{0\}$ is a weak supersolution of \eqref{dirichlet} with $f=0$ and $g\ge 0$ a.e. in $\R^N$, then $u>0$ in $\Omega$.
\end{thm}
\begin{proof}
We argue by contradiction, assuming that $u$ vanishes at some point of $\Omega$. We recall that, by Theorem \ref{wmp}, $u\ge 0$ in $\Omega$ and $u$ is lower semi-continuous. So, assuming without loss of generality that $\Omega$ is connected, the set
\[\Omega_+=\{x\in\Omega:\,u(x)>0\}\]
is open, nonempty and has a boundary in $\Omega$. Pick $x_1\in\partial\Omega_+\cap\Omega$ and set $\delta(x_1)=:2R>0$. By lower semi-continuity and $u\ge 0$, we get $u(x_1)=0$. We can find $x_0\in\Omega_+\cap B_R(x_1)$, and some $r\in(0,R)$ such that $u(x)\ge u(x_0)/2$ for all $x\in B_r(x_0)$. Let $\varphi\in H^s(\R^N)$ be as in Lemma \ref{subs}, and set for all $x\in\R^N$
\[w(x)=u(x)-\frac{u(x_0)}{2}\varphi(x-x_0).\]
It is easily seen that $w\in H^s(\R^N)$ is a weak supersolution of \eqref{dirichlet} in the domain $B_R(x_0)\setminus\overline B_r(x_0)$, with $g=0$. Hence, by Theorem \ref{wmp} we have $w\ge 0$ a.e. in $B_R(x_0)\setminus\overline B_r(x_0)$. In particular, noting that $x_1\in B_R(x_0)\setminus\overline B_r(x_0)$, we see that
\[u(x_1)\ge \frac{u(x_0)}{2}\varphi(x_1-x_0)\ge\frac{u(x_0)}{2}c(R-|x_1-x_0|)_+^s>0\]
by Lemma \ref{subs}, a contradiction.
\end{proof}

\begin{rem}
It is worth noting that strong maximum principle type results for the fractional Laplacian were already known. A statement for smooth $s$-harmonic functions can be found in \cite[Proposition 2.7]{CRS}. The strong maximum principle  was proved by Silvestre for distributional supersolutions but under a stronger semicontinuity and compactness condition, see \cite[Proposition 2.17]{S}. In \cite[Lemma 12]{LL} the strong maximum principle was proved for viscosity supersolutions of the fractional $p$-Laplacian in the case $s<1-1/p$. Recently in \cite[Theorem A.1]{BF} a weaker statement ($u>0$ {\em almost everywhere} without semicontinuity assumptions) has been proved through a logarithmic lemma for weak supersolutions of the fractional $p$-Laplacian.
\end{rem}

\noindent
We can now prove a fractional Hopf lemma. This has been first stated by Caffarelli, Roquejoffre \& Sire \cite[Proposition 2.7]{CRS} for smooth $s$-harmonic functions.

\begin{lem}\label{hl}
If $u\in H^s(\R^N)\setminus\{0\}$ is a weak supersolution of \eqref{dirichlet} with $f=0$ and $g\ge 0$ a.e. in $\R^N$, then there exists $C=C(u)>0$ such that $u(x)\ge C\delta(x)^s$ for all $x\in\overline\Omega$.
\end{lem}
\begin{proof}
Let
\[\Omega_h=\{x\in\overline\Omega:\,\delta(x)\le h\}.\] 
We know from Theorems \ref{wmp}, \ref{smp} that $u$ is lower semi-continuous and $u(x)>0$ in $\Omega$. Hence, by reducing $C>0$ if necessary, we only need to prove the lower bound on $\Omega_h$, where it holds
\begin{equation}
\label{defmh}
\inf_{\overline\Omega\setminus\Omega_h}u=m_h>0.
\end{equation}
By classical results (see Aikawa, Kiplel\"{a}inen, Shanmugalingam \& Zhong \cite{AKSZ}) we know that $C^{1,1}$-regularity of $\partial\Omega$ provides a uniform interior sphere condition. This in turn implies that there exists a sufficiently small $h>0$ such that if  $l\in (0,2h]$ and $x\in\Omega_{2h}$
\begin{equation}\label{distance}
\delta(x)=l \quad \Leftrightarrow \quad 
 B_l(x)\subseteq \Omega 
\end{equation}
and the metric projection $\Pi:\Omega_{2h}\to\partial\Omega$ is well defined. We fix such an $h$ and for arbitrary $x_0\in\Omega_h$ set $x_1=\Pi(x_0)$, $x_2=x_1-2h\nu(x_1)$, where $\nu:\partial\Omega\to\R^N$ is the outward unit vector. Then $\delta(x_2)\le 2h$ by construction and through \eqref{distance} we have $B_{2h}(x_2)\subseteq\Omega$, which forces $\delta(x_2)=2h$.
Let $\varphi\in H^s(\R^N)$ be defined as in Lemma \ref{subs} with $R=2h$ and $r=h$ and set $v(x)=m_h\varphi(x-x_2)$ as per \eqref{defmh}. For all $x\in\overline B_h(x_2)\subseteq \Omega\setminus\Omega_h$ we have
\[u(x)\ge m_h=v(x),\]
so $u-v$ is a weak supersolution of \eqref{dirichlet} in $B_{2h}(x_2)\setminus\overline B_h(x_2)$ with $f=g=0$. By Theorem \ref{wmp} we have $u\ge v$ in $B_{2h}(x_2)\setminus\overline B_h(x_2)$. In particular, we have
\[u(x_0)\ge v(x_0)\ge C\big(2h-|x_0-x_2|\big)^s=C\delta(x_0)^s,\]
with $C>0$ depending on $h$, $m_h$ and $\Omega$, which concludes the proof.
\end{proof}

\section{A priori bounds}\label{apb}

\noindent
In this section we prove some {\em a priori} bounds for the weak solutions of problem \eqref{bvp}, both in the subcritical and critical cases. We will use an adaptation of the classical Moser iteration technique. A similar method was used by Brasco, Lindgren \& Parini \cite[Theorem 3.3]{BLP} for the first eigenfunctions of the fractional Laplacian (in fact, for a more general, nonlinear operator, see Section \ref{sec6} below), while most $L^\infty$-bounds for nonlocal equations are based on a different method, see \cite{FP,ILPS,SV4}. A fractional version of De Giorgi's iteration method was developed by Mingione \cite{M}. We introduce some notation: for all $t\in\R$ and $k>0$, we set
\begin{equation}\label{trunk}
t_k={\rm sgn}(t)\min\{|t|,k\}.
\end{equation}
The Moser method in the fractional setting is based on the following elementary inequality:

\begin{lem}\label{ineq}
For all $a,b\in\R$, $r\ge 2$, and $k>0$ we have
\[(a-b)(a|a|_k^{r-2}-b|b|_k^{r-2})\geq \frac{4(r-1)}{r^2}(a|a|_k^{\frac{r}{2}-1}-b|b|_k^{\frac{r}{2}-1})^2.\]
\end{lem}
\begin{proof}
By the symmetry of the inequality, we may assume $a\ge b$. We set for all $t\in\R$
\[h(t)=\begin{cases}
{\rm sgn}(t)|t|^{\frac{r}{2}-1} & \text{if $|t|<k$} \\
\displaystyle\frac{2}{r}{\rm sgn}(t)k^{\frac{r}{2}-1} & \text{if $|t|\ge k$}.
\end{cases}\]
It is readily seen that
\[\int_b^a h(t)\,dt=\frac{2}{r}(a|a|_k^{\frac{r}{2}-1}-b|b|_k^{\frac{r}{2}-1})\]
and, since $4(r-1)\le r^2$, a similar computation gives
\[\int_b^a h(t)^2\,dt\le\frac{1}{r-1}(a|a|_k^{r-2}-b|b|_k^{r-2}).\]
Now, the Schwartz inequality yields
\[\Big(\int_b^a h(t)\,dt\Big)^2\le(a-b)\int_b^a h(t)^2\,dt,\]
which is the conclusion.
\end{proof}

\noindent
We prove an $L^\infty$-bound on the weak solutions of \eqref{bvp} (in the subcritical case such bound is uniform):

\begin{thm}\label{linfty}
If $f$ satisfies \eqref{gc}, then for any weak solution $u\in X(\Omega)$ of \eqref{bvp} we have $u\in L^\infty(\Omega)$. Moreover, if $q<2^*_s$ in \eqref{gc}, then there exists a function $M\in C(\R^+)$, only depending on the constants in \eqref{gc}, $N$, $s$ and $\Omega$, such that
\[\|u\|_\infty\le M(\|u\|_{2^*_s}).\]
\end{thm}
\begin{proof}
 Let $u\in X(\Omega)$ be a weak solution of \eqref{bvp} and set $\gamma=(2^*_s/2)^{1/2}$. For all $r\ge 2$, $k>0$, the mapping $t\mapsto t|t|_k^{r-2}$ is Lipschitz in $\R$, hence $u|u|_k^{r-2}\in X(\Omega)$. We apply the fractional Sobolev inequality, Lemma \ref{ineq}, test \eqref{weak} with $u|u|_k^{r-2}$, and we use \eqref{gc} to obtain
\begin{equation}\label{sc1}
\begin{split}
\|u|u|_k^{\frac{r}{2}-1}\|_{2^*_s}^2 &\le \|u|u|_k^{\frac{r}{2}-1}\|_X^2\le \frac{Cr^2}{r-1}\langle u,u|u|_k^{r-2}\rangle_X \\
&\le Cr\int_\Omega |f(x,u)||u||u|_k^{r-2}\, dx \\
&\le Cr\int_\Omega\big(|u||u|_k^{r-2}+|u|^q|u|_k^{r-2}\big)\, dx,
\end{split}
\end{equation}
for some $C>0$ independent of $r\geq 2$ and $k>0$. Applying the Fatou Lemma as $k\to\infty$ yields
\begin{equation}\label{sc2}
\|u\|_{\gamma^2 r}\le Cr^\frac{1}{r}\Big(\int_\Omega\big(|u|^{r-1}+|u|^{r+q-2}\big)\, dx\Big)^\frac{1}{r}
\end{equation}
(where the right hand side may be $\infty$). Our aim is to develop from \eqref{sc2} a suitable bootstrap argument to prove that $u\in L^p(\Omega)$ for all $p\ge 1$. We define recursively a sequence $\{r_n\}$ by choosing $\mu>0$ and setting
\[r_0=\mu, \quad r_{n+1}=\gamma^2 r_n+2-q.\]
The only fixed point of $t\to\gamma^2 t+2-q$ is
\[\mu_0=\frac{q-2}{\gamma^2-1},\]
so we have $r_n\to+\infty$ iff $\mu>\mu_0$. We now split the proof into the subcritical and critical cases.
\vskip4pt
\noindent
$\bullet$ {\em Subcritical case: $q<2^*_s$}. We fix
\begin{equation}\label{sc3}
\mu=2^*+2-q>\max\{2,\mu_0\},
\end{equation}
and bootstrap on the basis of \eqref{sc2}. Since  $r_0+q-2=2^*_s$, we have $u\in L^{r_0+q-2}(\Omega)$ (in particular $u\in L^{r_0-1}(\Omega)$). Hence, choosing $r=r_0$ in \eqref{sc2}, we obtain a finite right hand side, so $u\in L^{\gamma^2 r_0}(\Omega)=L^{r_1+q-2}(\Omega)$, and so on. Iterating this argument and noting that $r\mapsto r^{1/r}$ is bounded in $[2,\infty)$, for all $n\in\N$ we have $u\in L^{\gamma^2 r_n}(\Omega)$ and
\[\|u\|_{\gamma^2 r_n}\le H(n,\|u\|_{2^*_s})\]
(henceforth, $H$ will denote a continuous function of one or several real variables, whose definition may change case by case). By \eqref{sc3} we know that $\gamma^2 r_n\to\infty$ as $n\to\infty$, so for all $p\ge 1$ we can find $n\in\N$ such that $\gamma^2 r_n\ge p$. Applying H\"older inequality, for all $p\ge 1$ we have $u\in L^p(\Omega)$ and
\begin{equation}\label{sc4}
\|u\|_p\le H(p,\|u\|_{2^*_s}).
\end{equation}
The $L^p$-bound above is not yet enough to prove our assertion, as the right hand side may not be bounded as $p\to\infty$. Thus, we need to improve \eqref{sc4} to a {\em uniform} $L^p$-bound. Fix $\gamma'=\gamma/(\gamma-1)$ and notice that from \eqref{sc4} and H\"older inequality it follows
\[
\|1+|u|^{q-1}\|_{\gamma'}\le H(\|u\|_{2^*_s}).
\]
Therefore, for any $r\ge 2$ we have
\begin{align*}
\int_\Omega\big(|u|^{r-1}+|u|^{r+q-2}\big)\, dx &\le \|1+|u|^{q-1}\|_{\gamma'}\||u|^{r-1}\|_\gamma \le H(\|u\|_{2^*_s})\|u\|_{\gamma(r-1)}^{r-1} \\
&\le H(\|u\|_{2^*_s})|\Omega|^\frac{1}{\gamma r}\|u\|_{\gamma r}^{r-1}.
\end{align*}
Noting that $r\mapsto|\Omega|^{1/(\gamma r)}$ is bounded in $[2,\infty)$, we see that
\[\int_\Omega\big(|u|^{r-1}+|u|^{r+q-2}\big)\, dx\le H(\|u\|_{2^*_s})\|u\|_{\gamma r}^{r-1}.\]
The inequality above can be used in \eqref{sc2} to obtain the following estimate:
\[\|u\|_{\gamma^2 r}^r\le H(\|u\|_{2^*_s})\|u\|_{\gamma r}^{r-1}.\]
Setting $v=u/H(\|u\|_{2^*_s})$ and $r=\gamma^{n-1}$ ($\gamma^{n-1}\ge 2$ for $n\in\N$ big enough), we have the following nonlinear recursive relation:
\[
\|v\|_{\gamma^{n+1}}\le\|v\|_{\gamma^n}^{1-\gamma^{1-n}}
\]
which, iterated, provides
\[\|v\|_{\gamma^n}\le\|v\|_\gamma^{\Pi_{i=0}^{n-2}(1-\gamma^{-i})}\qquad n\in \N.\]
It is easily seen that the sequence $(\Pi_{i=0}^{n-2}(1-\gamma^{-i}))$ is bounded in $\R$, so for all $n\in\N$ we have
\[\|v\|_{\gamma^n}\le H(\|u\|_{2^*_s}).\]
Going back to $u$, and recalling that $\gamma^n\to\infty$ as $n\to\infty$, we find $M\in C(\R_+)$ such that for all $p\ge 1$
\[\|u\|_p\le M(\|u\|_{2^*_s}),\]
i.e., from classical results in functional analysis, $u\in L^\infty(\Omega)$ and
\begin{equation}\label{sc7}
\|u\|_\infty\le H(\|u\|_{2^*_s}).
\end{equation}
\vskip4pt
\noindent
$\bullet$ {\em Critical case: $q=2^*_s$.}
We start from \eqref{sc1}, with $r=q+1>2$, and fix $\sigma>0$ such that $Cr\sigma<1/2$. Then there exists $K_0>0$ (depending on $u$) such that
\begin{equation}
\label{piccolo}
\left(\int_{\{|u|> K_0\}}|u|^q\, dx\right)^{1-\frac 2 q}\leq\sigma.
\end{equation}
By H\"older inequality and \eqref{piccolo} we have
\begin{align*}
\int_\Omega |u|^q|u|_k^{r-2} \, dx&\le K_0^{q+r-2}|\{|u|\le K_0\}|+\int_{\{|u|>K_0\}}|u|^q|u|_k^{r-2}\, dx \\
&\le K_0^{q+r-2}|\Omega|+\Big(\int_\Omega (u^2|u|_k^{r-2})^{\frac q 2}\, dx\Big)^{\frac 2 q}\Big(\int_{\{|u|>K_0\}}|u|^q \, dx\Big)^{1-\frac 2 q} \\
&\le K_0^{q+r-2}|\Omega|+\sigma\|u|u|_k^{\frac{r}{2}-1}\|_q^2.
\end{align*}
Recalling that $Cr\sigma<1/2$, and that \eqref{sc1} holds, we obtain
\[\frac{1}{2}\|u|u|_k^\frac{q-1}{2}\|_q^2\le Cr\big(\|u\|_q^q+K_0^{2q-1}|\Omega|\big).\]
Letting $k\to\infty$, we have
\[\|u\|_\frac{q(q+1)}{2}\le \tilde H(K_0,\|u\|_q)\]
(where, as above, $\tilde H$ is a continuous function). Now the bootstrap argument can be applied through \eqref{sc2}, starting with 
\[r_0=\mu=\frac{q(q+1)}{2}+2-q>\mu_0=2,\]
since $u\in L^{r_0+q-2}(\Omega)$. The rest of the proof follows {\em verbatim}, providing in the end $u\in L^\infty(\Omega)$ and
\begin{equation}\label{sc9}
\|u\|_\infty\leq\tilde M(K_0, \|u\|_{2^*_s})
\end{equation}
for a convenient function $\tilde M\in C(\R^2)$.
\end{proof}

\begin{rem}
In the critical case $q=2^*_s$, the uniform $L^\infty$-estimate \eqref{sc7} cannot hold true. We introduce the {\em fractional Talenti functions} by setting for all $\eps>0$ and $z\in\R^N$
\[{\mathcal T}_{\eps,z}(x)=\Big(\frac{\varepsilon}{\varepsilon^2+|x-z|^2}\Big)^\frac{N-2s}{2}.\]
It is readily seen that there exists $\Gamma(N,s)>0$ such that, for all $\eps>0$ and $z\in\R^N$, $\Gamma(N,s)\mathcal{T}_{\eps,z}$ is a positive solution of the fractional equation
\begin{equation}
\label{int-crit}
(-\Delta )^s u=u^{\frac{N+2s}{N-2s}} \,\,\,\quad\text{in $\R^N$,}
\end{equation}
Actually, in the local case $s=1$, Chen, Li \& Ou \cite{CLO} have proved that $\mathcal{T}_{\eps,z}$ are the {\em only} positive solutions of \eqref{int-crit}. We have $\|{\mathcal T}_{\eps,z}\|_{\infty}\to\infty$ as $\eps\to 0$ and,
by rescaling, it follows that $\|{\mathcal T}_{\eps,z}\|_{2^*_s}$ is independent of $\eps$. If $z\in\Omega$,
$\eps$ is very small (so that almost all the mass of ${\mathcal T}_{\eps,z}$ is contained in $\Omega$)
and we truncate ${\mathcal T}_{\eps,z}$ so that it is set equal to zero outside $\Omega$, we would find that
\eqref{sc7} is violated as $\eps\to 0$. Thus, it seems that the non-uniform estimate \eqref{sc9}, involving a real number $K_0>0$ such that \eqref{piccolo} holds for a convenient $\sigma>0$, cannot be improved in general.
\end{rem}

\section{Proof of Theorem~\ref{min-local}}\label{s3}

\noindent
{\bf Proof that $(i)$ implies $(ii)$.} We shall divide the proof into several steps:
\vskip2pt
\noindent
\underline{\em Case $u_0=0$.}
We note that $\Phi(u_0)=0$, so our hypothesis rephrases as
\begin{equation}\label{abs0}
\inf_{u\in X(\Omega)\cap\overline B_\rho^\delta(0)}\Phi(u)=0.
\end{equation}
Again, we consider separately the subcritical and critical cases.
\vskip4pt
\noindent
$\bullet$ {\em Subcritical case: $q<2^*_s$.}
We argue by contradiction, assuming that there exists a sequence $(\eps_n)$ in $(0,\infty)$ such that $\eps_n\to 0$ and for all $n\in\N$
\[\inf_{u\in\overline B_{\eps_n}^X(0)}\Phi(u)=m_n<0.\]
By \eqref{gc} and the compact embedding $X(\Omega)\hookrightarrow L^q(\Omega)$, the functional $\Phi$ is sequentially weakly lower semicontinuous in $X(\Omega)$, hence $m_n$ is attained at some $u_n\in \overline B_{\eps_n}^X(0)$ for all $n\in\N$. We claim that, for all $n\in\N$, there exists $\mu_n\le 0$ such that for all $v\in X(\Omega)$
\begin{equation}\label{lag}
\langle u_n,v\rangle_X-\int_\Omega f(x,u_n)v \, dx=\mu_n\langle u_n,v\rangle_X.
\end{equation}
Indeed, if $u_n\in B_{\eps_n}^X(0)$, then $u_n$ is a local minimizer of $\Phi$ in $X(\Omega)$, hence a critical point, so \eqref{lag} holds with $\mu_n=0$. If $u_n\in\partial B_{\eps_n}^X(0)$, then $u_n$ minimizes $\Phi$ restricted to the $C^1$-Banach manifold
\[\Big\{u\in X(\Omega):\,\frac{\|u\|_X^2}{2}=\frac{\eps_n^2}{2}\Big\},\]
so we can find a Lagrange multiplier $\mu_n\in\R$ such that \eqref{lag} holds. More precisely, testing \eqref{lag} with $-u_n$ and recalling that $\Phi(u)\ge\Phi(u_n)$ for all $u\in B_{\eps_n}^X(0)$, we easily get
\[0\le\Phi'(u_n)(-u_n)=-\mu_n\|u_n\|_X^2,\]
hence $\mu_n\le 0$.
\vskip2pt
\noindent
Setting $C_n=(1-\mu_n)^{-1}\in (0,1]$, we see that for all $n\in\N$ the function $u_n\in X(\Omega)$ is a weak solution of the auxiliary boundary value problem
\[\begin{cases}
(- \Delta)^s\, u =C_n f(x,u) & \text{in } \Omega \\
u = 0 & \text{in } \R^N \setminus \Omega,
\end{cases}\]
where the nonlinearity satisfies \eqref{gc} uniformly with respect to $n\in\N$. By Theorem \ref{linfty} (and recalling that $(u_n)$ is bounded in $L^{2^*_s}(\Omega)$), there exists $M>0$ such that for all $n\in\N$ we have $u_n\in L^\infty(\Omega)$ with $\|u_n\|_\infty\le M$. This, in turn, implies that for all $n\in\N$
\[\|C_n f(\cdot,u_n(\cdot))\|_\infty\le a(1+M^{q-1}).\]
Now we apply Theorem \ref{ros}, which assures the existence of $\alpha>0$ and $C>0$ such that, for all $n\in\N$, we have $u_n\in C^{0,\alpha}_\delta(\overline\Omega)$ with $\|u_n\|_{\alpha,\delta}\le Ca(1+M^{q-1})$. By the compact embedding $C^{0,\alpha}_\delta(\overline\Omega)\hookrightarrow C^0_\delta(\overline\Omega)$, up to a subsequence, we see that $(u_n)$ is strongly convergent in $C^0_\delta(\overline\Omega)$, hence (by a simple computation) $(u_n)$ is uniformly convergent in $\overline\Omega$. Since $u_n\to 0$ in $X(\Omega)$, passing to a subsequence, we may assume $u(x)\to 0$ a.e. in $\Omega$, so we deduce $u_n\to 0$ in $C^0_\delta(\overline\Omega)$. In particular, for $n\in\N$ big enough we have $\|u_n\|_{0,\delta}\le\rho$ together with
\[\Phi(u_n)=m_n<0,\]
a contradiction to \eqref{abs0}.
\vskip4pt
\noindent
$\bullet$ {\em Critical case: $q=2^*_s$.}
We need to overcome a twofold difficulty, as the critical growth both prevents compactness (and hence the existence of minimizers of $\Phi$ on closed balls of $X(\Omega)$), and does not allow to get immediately a uniform estimate on the $L^\infty$-norms of solutions of the auxiliary problem. Again we argue by contradiction, assuming that there exist sequences $(\eps_n)$ in $(0,\infty)$ and $(w_n)$ in $X(\Omega)$ such that for all $n\in\N$ we have $w_n\in\overline B_{\eps_n}^X(0)$ and $\Phi(w_n)<0$. For all $k>0$ we define $f_k,F_k:\overline\Omega\times\R\to\R$ by setting for all $(x,t)\in\overline\Omega\times\R$
\[f_k(x,t)=f(x,t_k), \quad F_k(x,t)=\int_0^t f_k(x,\tau)\,d\tau\]
($t_k$ defined as in \eqref{trunk}). Accordingly, we define the functionals $\Phi_k\in C^1(X(\Omega))$ by setting for all $u\in X(\Omega)$
\[
\Phi_k(u)=\frac{\|u\|_X^2}{2}-\int_\Omega F_k(x,u)\, dx.
\]
By the dominated convergence Theorem, for all $u\in X(\Omega)$ we have $\Phi_k(u)\to\Phi(u)$ as $k\to\infty$. So, for all $n\in\N$ we can find $k_n\ge 1$ such that $\Phi_{k_n}(w_n)<0$. Since $f_k$ has subcritical growth, for all $n\in\N$ there exists $u_n\in\overline B_{\eps_n}^X(0)$ such that
\[\Phi_{k_n}(u_n)=\inf_{u\in \overline B_{\eps_n}^X(0)}\Phi_{k_n}(u)\le\Phi_{k_n}(w_n)<0.\]
As in the previous case we find a sequence $(C_n)$ in $(0,1]$ such that $u_n$ is a weak solution of
\[\begin{cases}
(- \Delta)^s\, u =C_n f_{k_n}(x,u) & \text{in } \Omega \\
u = 0 & \text{in } \R^N \setminus \Omega,
\end{cases}\]
and the nonlinearities $C_n f_{k_n}$ satisfy \eqref{gc} uniformly with respect to $n\in\N$. We recall that $u_n\to 0$ in $X(\Omega)$, hence in $L^{2^*_s}(\Omega)$. So, \eqref{piccolo} holds with $K_0=0$ and $n\in\N$ big enough. Therefore, Theorem \ref{linfty} assures that $u_n\in L^\infty(\Omega)$ and that $\|u_n\|_\infty\le M$ for some $M>0$ independent of $n\in\N$. Now we can argue as in the subcritical case, proving that (up to a subsequence) $u_n\to 0$ in $C^0_\delta(\overline\Omega)$ and uniformly in $\overline\Omega$. In particular, for $n\in\N$ big enough we have $\|u_n\|_{0,\delta}\le\rho$ and $\|u_n\|_\infty\le 1$, hence
\[\Phi(u_n)=\Phi_{k_n}(u_n)<0,\]
a contradiction to \eqref{abs0}.
\vskip2pt
\noindent
\underline{\em Case $u_0\ne 0$.}
For all $v\in C^\infty_c(\Omega)$, we have in particular $v\in X(\Omega)\cap C^0_\delta(\overline\Omega)$, so the minimality ensures
\begin{equation}\label{crip}
\Phi'(u_0)(v)=0,\quad  v\in C^\infty_c(\Omega).
\end{equation}
Since $C^\infty_c(\Omega)$ is a dense subspace of $X(\Omega)$ (see Fiscella, Servadei \& Valdinoci \cite{FSV}) and $\Phi'(u_0)\in X(\Omega)^*$, equality \eqref{crip} holds in fact for all $v\in X(\Omega)$, i.e., $u_0$ is a weak solution of \eqref{bvp}. By Theorem \ref{linfty}, we have $u_0\in L^\infty(\Omega)$, hence $f(\cdot,u_0(\cdot))\in L^\infty(\Omega)$. Now Theorem \ref{ros} implies $u_0\in C^0_\delta(\overline\Omega)$. We set for all $(x,t)\in\Omega\times\R$
\[\tilde F(x,t)=F(x,u_0(x)+t)-F(x,u_0(x))-f(x,u_0(x))t,\]
and for all $v\in X(\Omega)$
\[
\tilde\Phi(v)=\frac{\|v\|_X^2}{2}-\int_\Omega \tilde F(x,v)\, dx.
\]
Clearly we have $\tilde\Phi\in C^1(X(\Omega))$ and the mapping $\tilde f:\overline\Omega\times\R\to\R$ defined by $\tilde f(x,t)=\partial_t \tilde F(x,t)$ satisfies a growth condition of the type \eqref{gc}. Besides, by \eqref{crip}, we have for all $v\in X(\Omega)$
\begin{align*}
\tilde\Phi(v) &= \frac{1}{2}\big(\|u_0+v\|_X^2-\|u_0\|_X^2\big)-\int_\Omega\big(F(x,u_0+v)-F(x,u_0)\big)\, dx \\
&= \Phi(u_0+v)-\Phi(u_0),
\end{align*}
in particular $\tilde\Phi(0)=0$. Our hypothesis thus rephrases as
\[\inf_{v\in X(\Omega)\cap\overline B_\rho^\delta(0)}\tilde\Phi(v)=0\]
and by the previous cases, we can find $\eps>0$ such that for all $v\in X(\Omega)$, $\|v\|_X\le\eps$, we have $\tilde\Phi(v)\ge 0$, namely $\Phi(u_0+v)\ge\Phi(u_0)$. 
\vskip4pt
\noindent
{\bf Proof that $(ii)$ implies $(i)$.}
Suppose by contradiction that there exists a sequence $(u_n)$ which converges to $u$ in $C^0_\delta(\overline{\Omega})$ and
$\Phi(u_n)<\Phi(u_0)$. Observe that 
$$
\int_{\Omega} F(x,u_n)\, dx \to \int_{\Omega} F(x,u)\, dx,
$$
and this, together with $\Phi(u_n)<\Phi(u_0)$, implies that
\begin{equation}
\label{sup-lim}
\limsup_n \|u_n\|_{X}^2\leq \|u\|_X^2.
\end{equation}
In particular $(u_n)$ is bounded in $X(\Omega)$ and, up to a subsequence, it converges weakly 
and pointwisely to $u_0$. By semicontinuity, \eqref{sup-lim} forces $\|u_n\|_X\to \|u_0\|_X$, thus $u_n\to u_0$
in $X$ as $n\to\infty$, which concludes the proof.
\qed

\section {Applications}\label{s4}

\noindent
In this section we present some existence and multiplicity results for the solutions of problem \eqref{bvp}, under \eqref{gc} plus some further conditions. In the proofs of such results, Theorem \ref{min-local} will play an essential r\^ole.
\vskip2pt
\noindent
Our first result ensures that, if problem \eqref{bvp} admits a weak subsolution and a weak supersolution, then it admits a solution which is also a local minimizer of the energy functional. We define weak super- and subsolutions of \eqref{bvp} as in Section \ref{s2}.

\begin{thm}\label{subsuper}
Let $f:\Omega\times\R\to \R$ be a Carath\'eodory function satisfying \eqref{gc} and $f(x,\cdot)$ be nondecreasing in $\R$ for a.a. $x\in\Omega$. Suppose that $\overline u,\underline u\in H^s(\R^N)$ are a weak supersolution and a weak subsolution, respectively, of \eqref{bvp} which are not solutions. Then, there exists a solution $u_0\in X(\Omega)$ of \eqref{bvp} such that $\underline u\le u_0\le \overline u$ a.e. in $\Omega$ and $u_0$ is a local minimizer of $\Phi$ on $X(\Omega)$.
\end{thm}
\begin{proof}
We first observe that $\underline u\le\overline u$ a.e. in $\R^N$. Indeed, by monotonicity of $f(x,\cdot)$, $\overline u-\underline u$ is easily seen to be a weak supersolution of \eqref{dirichlet} with $f=g=0$ and Theorem \ref{wmp} forces $\overline u-
\underline u\ge 0$. We set for all $(x,t)\in\Omega\times\R$
\[
\tilde f(x,t):=\begin{cases}
f(x,\underline u(x)) & \text{if $t\le\underline u(x)$} \\
f(x,t) & \text{if $\underline u(x)<t<\overline u(x)$} \\
f(x,\overline u(x)) & \text{if $t\ge\overline u(x)$}
\end{cases}
\quad \tilde F(x,t):=\int_0^t \tilde f(x,\tau)\,d\tau\]
and for all $u\in X(\Omega)$
\[
\tilde\Phi(u):=\frac{\|u\|_X^2}{2}-\int_\Omega \tilde F(x,u)\, dx.
\]
The functional $\tilde\Phi\in C^1(X(\Omega))$ is sequentially weakly lower semicontinuous and coercive, since monotonicity of $f(x,\cdot)$, \eqref{gc} and H\"older inequality imply for all $u\in X(\Omega)$
\[
\int_\Omega \tilde F(x,u)\, dx \le \int_\Omega \big(|f(x,\underline u)|+|f(x,\overline u)|\big)|u|\, dx \le C(1+\|\underline u\|_q^{q-1}+\|\overline u\|_q^{q-1})\|u\|_X.
\]
Let $u_0\in X(\Omega)$ be a global minimizer of $\tilde\Phi$, which therefore solves
\[\begin{cases}
(- \Delta)^s\, u_0 =\tilde f(x,u_0) & \text{in } \Omega \\
u_0= 0 & \text{in } \R^N \setminus \Omega.
\end{cases}\]
Again by monotonicity and the definition of $\tilde f$, we have, in the weak sense,
\[(-\Delta)^s(\overline u-u_0)\ge f(x,\overline u)-\tilde f(x,u_0)\ge 0\]
in $\Omega$, while $\overline u-u_0\ge 0$ in $\R^N\setminus\Omega$, so $\overline u-u_0$ is a weak supersolution of \eqref{dirichlet}, nonnegative by Theorem \ref{wmp}. It holds $\overline u-u_0\neq 0$, otherwise we would have $\overline u\in X(\Omega)$ and, in the weak sense,
\[(-\Delta)^s\overline u=\tilde f(x,\overline u)=f(x,\overline u)\]
in $\Omega$, against our hypotheses on $\overline u$. By Lemma \ref{hl}, we have $(\overline u-u_0)/\delta^s\ge C$ in $\overline\Omega$ for some $C>0$. Similarly we prove that $(u_0-\underline u)/\delta^s\ge C$ in $\overline\Omega$. Thus, $u_0$ is a solution of \eqref{bvp}.
\vskip2pt
\noindent
Now we prove that $u_0$ is a local minimizer of $\Phi$. By Theorems \ref{linfty} and \ref{ros} we have $u_0\in C^0_\delta(\overline\Omega)$. For any $u\in \overline B_{C/2}^\delta(u_0)$ we have in $\overline\Omega$
\[\frac{\overline u-u}{\delta^s}=\frac{\overline u-u_0}{\delta^s}+\frac{u_0-u}{\delta^s}\ge C-\frac{C}{2},\]
in particular $\overline u-u>0$ in $\Omega$. Similarly, $u-\underline u>0$ in $\Omega$, so $\tilde\Phi$ agrees with $\Phi$ in $\overline B_{C/2}^\delta(u_0)\cap X(\Omega)$ and $u_0$ turns out to be a local minimizer of $\Phi$ in $C^0_\delta(\overline\Omega)\cap X(\Omega)$. Now, Theorem \ref{min-local} implies that $u_0$ is a local minimizer of $\Phi$ in $X(\Omega)$ as well.
\end{proof}

\noindent
We present now a multiplicity theorem for problem \eqref{bvp}, whose proof combines Theorem \ref{min-local}, spectral properties of $(-\Delta)^s$ and Morse-theoretical methods (the fully nonlinear case is examined in \cite[Theorem 5.3]{ILPS}). In what follows, $0<\lambda_{1,s}<\lambda_{2,s}\le\ldots$ will denote the eigenvalues of $(-\Delta)^s$ in $X(\Omega)$ (see \cite{SV1}).

\begin{thm}\label{app1}
Let $f:\Omega\times\R\to\R$ be a Carath\'eodory function satisfying
\begin{itemize}
\item[$(i)$] $|f(x,t)|\le a(1+|t|^{q-1})$ a.e. in $\Omega$ and for all $t\in\R$ ($a>0$, $1<q<2^*_s$);
\item[$(ii)$] $f(x,t)t\ge 0$ a.e. in $\Omega$ and for all $t\in\R$;
\item[$(iii)$] $\displaystyle\lim_{t\to 0}\frac{f(x,t)-b|t|^{r-2}t}{t}=0$ uniformly a.e. in $\Omega$ ($b>0$, $1<r<2$);
\item[$(iv)$] $\displaystyle\limsup_{|t|\to\infty}\frac{2F(x,t)}{t^2}<\lambda_{1,s}$ uniformly a.e. in $\Omega$.
\end{itemize}
Then problem \eqref{bvp} admits at least three non-zero solutions.
\end{thm}
\begin{proof}
We define $\Phi\in C^1(X(\Omega))$ as in the Introduction. From $(ii),(iii)$ we immediately see that $0$ is a critical point of $\Phi$, which is not a local minimizer by \cite[Lemma 5.5]{ILPS}. We introduce two truncated energy functionals, setting for all $(x,t)\in\Omega\times\R$
\[f_\pm(x,t)=f(x,\pm t_\pm),\quad F_\pm(x,t)=\int_0^t f_\pm(x,\tau)\,d\tau\]
and for all $u\in X(\Omega)$
\[
\Phi_\pm(u)=\frac{\|u\|_X^2}{2}-\int_\Omega F_\pm(x,u)\, dx.
\]
Clearly $f_+$ satisfies \eqref{gc}. It can be easily seen (see \cite[Lemma 5.5]{ILPS}) that there exists $u^+\in X(\Omega)\setminus\{0\}$ such that
\[
\Phi_+(u^+)=\inf_{u\in X(\Omega)}\Phi_+(u).
\]
Then, taking into account Theorem~\ref{wmp} and $(ii)$, $u^+$ is a nonnegative weak solution to~\eqref{bvp}. By Theorem \ref{linfty}, we have $u^+\in L^\infty(\Omega)$, so by Theorem \ref{ros} we deduce $u^+\in C^0_\delta(\overline\Omega)$. Moreover, again by $(ii)$, $u^+$ is a weak supersolution of problem \eqref{dirichlet} with $f=g=0$, hence by Lemma \ref{hl} $u^+/\delta^s>0$ in $\overline\Omega$. Now \cite[Lemma 5.1]{ILPS} implies that $u^+\in{\rm int}(C_+)$, where
\[C_+=\{u\in C^0_\delta(\overline\Omega):\, u(x)\ge 0 \ \mbox{in $\overline\Omega$}\}\]
and the interior is defined with respect to the $C^0_\delta(\overline\Omega)$-topology. Let $\rho>0$ be such that $B^\delta_\rho(u^+)\subset C_+$. Since $\Phi$ and $\Phi_+$ agree on $C_+\cap X(\Omega)$, 
\[
\Phi(u^++v)\ge\Phi(u^+),\qquad v\in B^\delta_\rho(0)\cap X(\Omega)
\]
and by Theorem \ref{min-local}, $u^+$ is a strictly positive local minimizer for $\Phi$ in $X(\Omega)$. Similarly, looking at $\Phi_-$, we can detect another strictly negative local minimizer $u^-\in -{\rm int}(C_+)$ of $\Phi$. Now, a Morse-theoretic argument shows that there exists a further critical point $\tilde u\in X(\Omega)$ of $\Phi$ with $u\notin\{0,u^\pm\}$ (see the proof of \cite[Theorem 5.3]{ILPS}).
\end{proof}

\noindent
We conclude this section with a fractional version of a classical multiplicity result for semilinear problems based on Morse theory:

\begin{thm}\label{app2}
Let $f\in C^1(\R)$ satisfy
\begin{itemize}
\item[$(i)$] $|f'(t)|\le a(1+|t|^{q-2})$ for all $t\in\R$ ($a>0$, $1<q\le 2^*_s$);
\item[$(ii)$] $f(t)t\ge 0$ and for all $t\in\R$;
\item[$(iii)$] $f'(0)>\lambda_{2,s}$ and $f'(0)$ is not an eigenvalue of $(-\Delta)^s$ in $X(\Omega)$;
\item[$(iv)$] $\displaystyle\limsup_{|t|\to\infty}\frac{f(t)}{t}<\lambda_{1,s}$.
\end{itemize}
Then problem \eqref{bvp} admits at least four non-zero solutions.
\end{thm}
\begin{proof}
Due to $(i)$, we have $\Phi\in C^2(X(\Omega))$, and by $(iv)$ $\Phi$ is coercive. By $(iii)$, we know that $0$ is a nondegenerate critical point of $\Phi$ with Morse index $m\ge 2$ (see Li, Perera \& Su \cite[Proposition 1.1]{LPS}). Therefore, reasoning as in the proof of Theorem \ref{app1}, we find two local minimizers $u^\pm\in\pm C_+$ for $\Phi$, with $u^+>0$ and $u^-<0$ in $\Omega$. Now, the Hess-Kato Theorem and a Morse-theoretic argument provide two further critical points $u_0,u_1\in X(\Omega)\setminus\{0,u^\pm\}$ (as in Liu \& Liu \cite[Theorem 1.3]{LL1}).
\end{proof}

\section{Final comments and open questions}\label{sec6}

\noindent
Let $p \in (1,\infty)$ and $s \in (0,1)$.
Recently, in \cite{ILPS}, quasi-linear problems involving the {\em fractional $p$-Laplacian operator} were investigated via techniques of Morse theory applied to the functional
$$
\Phi(u)=\frac{1}{p}\int_{\R^{2N}}\frac{|u(x)-u(y)|^p}{|x-y|^{N+ps}}\, dx\, dy-\int_\Omega F(x,u)\, dx,
$$
over the space of functions $u\in W^{s,p}(\R^N)$ with $u=0$ outside $\Omega$. Critical points of $\Phi$ give rise to nonlinear equations whose leading term is the {\em fractional $p$-Laplacian}, namely (up to a multiplicative constant)
\[
(-\Delta)^s_p u(x)= \lim_{\eps\to 0^+}\int_{\R^N\setminus B_\eps(x)}\frac{|u(x)-u(y)|^{p-2}(u(x)-u(y))}{|x-y|^{N+ps}}\, dy.
\]
Recent contributions on the subject of the fractional $p$-Laplacian operator are also contained in \cite{C,DKP,FP,LL}. 
\vskip2pt
\noindent
A natural question is whether a counterpart of Theorem \ref{min-local} holds in this nonlinear setting. This would provide a nonlocal version of the results of Garc\`{\i}a Azorero, Peral Alonso \& Manfredi \cite{GPM}, which extend the Brezis-Nirenberg theorem on local minimizers to nonlinear operators of the $p$-Laplacian type. Notice that the Moser iteration used in the proof of Theorem \ref{linfty} seems flexible enough to carry over in the nonlinear case (with \cite[Lemma C.2]{BLP} replacing Lemma \ref{ineq}). Hence, the main difficulty seems to be the proof of a boundary regularity estimate for the boundedly inhomogeneous fractional $p$-Laplacian equation as the one of Theorem \ref{ros}.
\vskip2pt
\noindent
Another point of interest lies in the fractional Hopf Lemma. As seen in Section \ref{s4}, the main point in focusing to $C^0_\delta(\overline\Omega)$ local minimizers is the fact that many order-related subsets of $X(\Omega)$ turn out to have nonempty interior with respect to the $C^0_\delta(\overline\Omega)$-topology. As mentioned in the Introduction, this is in strong contrast with the features of the topology of $X(\Omega)$, and the main tool to exploit this difference is Lemma \ref{hl}. It would be therefore interesting to explore the validity of such a statement for more general nonlocal operators, and for the fractional $p$-Laplacian in particular. 
\vskip2pt
\noindent
Finally, it is worth noting that in \cite{BN}, the sub-supersolution principle analogous to Theorem \ref{subsuper} is proved under a more general hypothesis on the nonlinearity $f(x, t)$, namely  
\[
\text{There exists $k\geq 0$ such that for a.e. $x\in \Omega$ the map $t\mapsto f(x, t)+kt$ is non-decreasing}.
\]
While we considered in Theorem \ref{subsuper} only non-decreasing nonlinearities, it seems that with little effort one can obtain the tools needed to treat the latter, more general, case. Indeed, it suffices to prove, for the operator $(-\Delta)^su+ku$,  $k\geq 0$, all the corresponding results of Section \ref{s2}.

\bigskip

\end{document}